\documentclass[11pt]{article}
\usepackage{graphicx,amsmath,amsfonts,amssymb,bbm,amsthm}
\usepackage{psfrag}
\usepackage{epsf}

\makeatletter\@addtoreset {equation}{section}\makeatother

\theoremstyle{plain}
\newtheorem{theorem}{Theorem}[section]
\newtheorem{lemma}[theorem]{Lemma}

\newtheorem{proposition}[theorem]{Proposition}
\newtheorem{corollary}[theorem]{Corollary}
\theoremstyle{remark}
\newtheorem{remark}[theorem]{Remark}

{\hspace*{\fill}$\rule{.3\baselineskip}{.35\baselineskip}$\end{trivlist}}

\newcommand{\R}{\mathbb{R}}

\newcommand{\Z}{\mathbb{Z}}

\renewcommand{\phi}{\varphi}

\newcommand{\TS}{\textstyle}
\newcommand{\dd}{\,\mathrm{d}}
\newcommand{\half}{\textstyle\frac12}
\newcommand{\MM}{\mathcal{M}}

\begin{document}

\title{\bf Orbital stability in the cubic defocusing NLS equation:
II. The black soliton}

\author{Thierry Gallay$^{1}$ and Dmitry Pelinovsky$^{2}$ \\
{\small $^{1}$ Institut Fourier, Universit\'e de Grenoble 1,
38402 Saint-Martin-d'H\`eres, France} \\
{\small $^{2}$ Department of Mathematics, McMaster
University, Hamilton, Ontario, Canada, L8S 4K1}  }

\date{\today}
\maketitle

\begin{abstract}
Combining the usual energy functional with a higher-order conserved
quantity originating from integrability theory, we show that the black
soliton is a local minimizer of a quantity that is conserved along the
flow of the cubic defocusing NLS equation in one space dimension. This
unconstrained variational characterization gives an elementary proof
of the orbital stability of the black soliton with respect to
perturbations in $H^2(\R)$.
\end{abstract}

\section{Introduction}
\label{sec:intro}

In this work we show how the techniques developed in the companion
paper \cite{GP} to investigate the stability properties of the cnoidal
periodic waves of the cubic defocusing nonlinear Schr\"odinger
equation in one space dimension can be extended to provide a new and
rather elementary proof of orbital stability in the limiting case of
the black soliton. We thus consider the cubic defocusing NLS equation
\begin{equation}
\label{nls}
  i \psi_t(x,t) + \psi_{xx}(x,t) - |\psi(x,t)|^2 \psi(x,t) \,=\, 0,
\end{equation}
where $\psi$ is a complex-valued function of $(x,t) \in \R \times
\R$. The black soliton is the particular solution of \eqref{nls}
given by $\psi(x,t) = e^{-it}u_0(x)$, where
\begin{equation}
\label{black-soliton}
  u_0(x) \,=\, \tanh\Bigl(\frac{x}{\sqrt{2}}\Bigr), \qquad x \in \R.
\end{equation}
For later use, we note that the soliton profile $u_0 : \R \to \R$
satisfies the differential equations
\begin{equation}
\label{wave}
  u_0' \,=\, \frac{1}{\sqrt{2}}\Bigl(1-u_0^2\Bigr), \qquad
  \hbox{hence}\qquad u_0'' + u_0 - u_0^3 \,=\, 0.
\end{equation}

The NLS equation \eqref{nls} has many symmetries and conserved
quantities, which play a crucial role in the dynamics of the system.
In particular, the gauge invariance $\psi \mapsto e^{i\theta}\psi$ and
the translation invariance $\psi \mapsto \psi(\cdot-\xi)$ give
rise to the conservation of the charge $Q$ and the momentum $M$,
respectively, where
\begin{equation}
\label{QMdef}
  Q(\psi) \,=\, \int_\R \Bigl(|\psi|^2 -1\Bigr)\dd x, \qquad
  M(\psi) \,=\, \frac{i}{2} \int_\R \Bigl(\bar{\psi} \psi_x -
  \psi \bar{\psi}_x\Bigr) \dd x.
\end{equation}
Since the NLS equation \eqref{nls} is an autonomous Hamiltonian
system, we also have the conservation of the energy
\begin{equation}
\label{energy}
  E(\psi) \,=\, \int_\R \left(|\psi_x|^2 + \frac{1}{2} (1 - |\psi|^2)^2
  \right) \dd x.
\end{equation}
In what follows, our goal is to study the stability of the black
soliton \eqref{black-soliton}, and we shall therefore restrict
ourselves to solutions of \eqref{nls} for which $|\psi| \to 1$ as $|x|
\to \infty$. This is why we defined the conserved quantities
\eqref{QMdef}, \eqref{energy} in such a way that the integrands vanish
when $|\psi| = 1$ and $\psi_x = 0$.

The nonlinear stability of the black soliton \eqref{black-soliton} has
been studied in several recent works. In \cite{BGSS} the authors apply
the variational method of Cazenave and Lions \cite{CL}, which relies
on the fact that the black soliton \eqref{black-soliton} is a global
minimizer of the energy $E$ for a fixed value of the momentum $M$. The
difficulty with this approach is that the momentum is not defined for
all finite-energy solutions, so that the integral defining $M$ in
\eqref{QMdef} has to be renormalized and properly interpreted. A
slightly different proof was subsequently given in \cite{GS}, in the
spirit of the work by Weinstein \cite{We} and Grillakis, Shatah, and
Strauss \cite{GSS}. The main idea is to show that the energy
functional \eqref{energy} becomes coercive in a neighborhood of the
black soliton \eqref{black-soliton} if the conservation of the
momentum is used to get rid of one unstable direction. Both results in
\cite{BGSS,GS} are variational in nature and establish orbital
stability of the black soliton in the energy space. Note that
asymptotic stability of the black soliton is also proved in \cite{GS},
using ideas and techniques developed by Martel and Merle for the
generalized Korteweg-de Vries equation \cite{MM}. In a different direction,
a more precise orbital stability result was obtained in \cite{GZ}
for sufficiently smooth and localized perturbations, using the inverse
scattering transform method which relies on the integrability of the
cubic defocusing NLS equation \eqref{nls}. Similarly, 
asymptotic stability of the black soliton and several dark solitons 
was recently proved in \cite{Cuccagna-private}.

As as consequence of integrability, the NLS equation \eqref{nls}
has many conserved quantities in addition to the charge, the
momentum, and the energy. In the present work, we introduce a
new variational approach based on the higher-order functional
\begin{equation}
\label{Sdef}
  S(\psi) \,=\, \int_{\R} \left( |\psi_{xx}|^2 + 3 |\psi|^2
  |\psi_x|^2 + \frac{1}{2} (\bar{\psi} \psi_x + \psi \bar{\psi}_x)^2
  + (1 - |\psi|^2)^2 \Bigl( 1 + \frac{1}{2} |\psi|^2 \Bigr) \right)
  \dd x,
\end{equation}
which is also conserved under the evolution defined by \eqref{nls}.
The latter claim can be proved by a straightforward but cumbersome
calculation, or by more educated techniques as described, e.g., in
\cite[Section~2.3]{Yang}. The natural domain of definition for the
functional \eqref{Sdef} is the $H^2$ energy space defined by
\begin{equation}
\label{Xdef}
  X \,=\, \Bigl\{\psi \in H^2_{\rm loc}(\R)\,: \quad \psi_x \in
  H^1(\R), ~1 - |\psi|^2 \in L^2(\R) \Bigr\}.
\end{equation}
Indeed, if $\psi \in X$, then $\zeta := 1 - |\psi|$ belongs to
$H^1(\R)$, because $|\zeta| \le |1 - |\psi|^2| \in L^2(\R)$ and
$\zeta_x = -|\psi|_x \in L^2(\R)$. By Sobolev's embedding of $H^1(\R)$
into $L^{\infty}(\R)$, we thus have $|\psi| = 1 - \zeta \in
L^\infty(\R)$, and from the definitions (\ref{Sdef}) and (\ref{Xdef}),
it follows easily that $S(\psi) < \infty$. Since $u_0'$, $u_0''$, and
$1 - u_0^2$ decay exponentially to zero as $|x| \to \infty$, it is
clear that $u_0 + H^2(\R) \subset X$, so that the functional
\eqref{Sdef} is well defined for $H^2$ perturbations of the soliton
profile $u_0$. This allows us to define the differential of $S$ at
$u_0$, and a direct calculation using the differential equations
\eqref{wave} reveals that $u_0$ is a {\em critical point} of $S$, in
the sense that $S'(u_0) = 0$.

Unfortunately, the second variation $S''(u_0)$ has no
definite sign \cite{GP}, hence it is not possible to prove orbital
stability of the black soliton using the functional $S$ alone.
As is explained in the companion paper \cite{GP}, which is devoted
to the stability of periodic waves for the NLS equation \eqref{nls},
it is possible to cure that problem by subtracting from $S$ an
appropriate multiple of the energy $E$, which is well defined
on $X$ and also satisfies $E'(u_0) = 0$. The optimal choice is
\begin{equation}
\label{Lamdef}
  \Lambda(\psi) \,=\, S(\psi) - 2 E(\psi), \qquad \psi \in X.
\end{equation}
We then have $\Lambda'(u_0) = 0$, and the starting point of our
approach is the following result, which asserts that the second
variation $\Lambda''(u_0)$ is nonnegative.

\begin{proposition}
\label{prop-main}
The second variation of the functional \eqref{Lamdef} at the black
soliton \eqref{black-soliton} is nonnegative for perturbations
in $H^2(\R)$.
\end{proposition}

It is important to realize that Proposition~\ref{prop-main} gives an
{\em unconstrained} variational characterization of the black soliton
$u_0$, which is our main motivation for introducing the higher-order
conserved quantity \eqref{Sdef}. In contrast, the approach in
\cite{BGSS,GS} relies on the fact that $u_0$ is a minimum of the
energy $E(\psi)$ subject to the constraint $\MM(\psi) = \MM(u_0)$,
where $\MM$ is a suitably renormalized version of the momentum $M$ defined
in \eqref{QMdef}. 

The proof of Proposition~\ref{prop-main} developed in Section
\ref{sec:positive} actually shows that the second variation
$\Lambda''(u_0)$ is positive except for degeneracies due to
symmetries: the nonnegative self-adjoint operator associated with
$\Lambda''(u_0)$ has a simple zero eigenvalue which is due to
translation invariance, and the essential spectrum extends all the way
to the origin due to gauge invariance.  As a consequence,
perturbations in $H^2(\R)$ can include slow modulations of the phase
of the black soliton far away from the origin, which hardly increase
the functional $\Lambda$. This means that the second variation
$\Lambda''(u_0)$ is not coercive in $H^2(\R)$, even if modulation
parameters are used to remove the zero modes due to the
symmetries. For that reason, we are not able to control the
perturbations of the black soliton in the topology of $H^2(\R)$, but
only in a weaker sense that allows for a slow drift of the phase at
infinity, see Section~\ref{sec:modulation} below for a more detailed
discussion.

To formulate our main result, we equip the space $X$ with the distance
\begin{equation}
\label{distance}
  d_R(\psi_1,\psi_2) \,=\, \|(\psi_1 - \psi_2)_x\|_{H^1(\R)}
  + \| |\psi_1|^2 - |\psi_2|^2\|_{L^2(\R)} + \| \psi_1-\psi_2\|_{L^2(-R,R)},
\end{equation}
where $R \ge 1$ is a parameter. Note that $d_R$ is the exact analogue,
at the $H^2$ level, of the distance that is used in previous
variational studies of the black soliton, including
\cite{BGSS,Gerard,GS}. As is easily verified, a function $\psi \in
H^2_{\rm loc}(\R)$ belongs to $X$ if and only if $d_R(\psi,u_0) <
\infty$; moreover, different choices of $R$ give equivalent distances
on $X$. To prove orbital stability of the black soliton with profile
$u_0$, the idea is to consider solutions $\psi$ of the NLS equation
\eqref{nls} for which $d_R(\psi,u_0)$ is small.  This is certainly the
case if $\|\psi-u_0\|_{H^2}$ is small, but the converse is not true
because $d_R(\psi,u_0)$ does not control the $L^2$ norm of the
difference $\psi - u_0$ on the whole real line.  We shall prove in
Section~\ref{sec:stability} that the distance $d_R$ is well adapted to the
functional $\Lambda$ near $u_0$, in the sense that
\begin{equation}
\label{Lamcoer}
  \Lambda(\psi) - \Lambda(u_0) \,\ge\, C d_R(\psi,u_0)^2
  \qquad \hbox{when}\quad d_R(\psi,u_0) \ll 1,
\end{equation}
provided the perturbation $\psi-u_0$ satisfies a pair of orthogonality
conditions. As is usual in orbital stability theory, these
orthogonality conditions can be fulfilled if we replace $\psi$ by
$e^{i\theta} \psi(\cdot+\xi)$ for some appropriate modulation
parameters $\theta,\xi \in \R$, see Section~\ref{sec:modulation}
below. It is then easy to deduce from \eqref{Lamcoer} that solutions
of the NLS equation \eqref{nls} with initial data $\psi$ satisfying
$d_R(\psi_0,u_0) \ll 1$ will stay close for all times to the orbit of
the black soliton under the group of translations and phase
rotations. The precise statement is:

\begin{theorem}
\label{theorem-soliton}
Fix $R \ge 1$ and let $u_0 \in X$ be the black soliton
\eqref{black-soliton}. Given any $\epsilon > 0$, there exists
$\delta > 0$ such that, for any $\psi_0 \in X$ satisfying
\begin{equation}
\label{bound-initial}
  d_R(\psi_0,u_0) \,\le\, \delta,
\end{equation}
the global solution $\psi(\cdot,t)$ of the NLS equation \eqref{nls}
with initial data $u_0$ has the following property. For any
$t \in \R$, there exist $\xi(t) \in \R$ and $\theta(t) \in
\R/(2\pi\Z)$ such that
\begin{equation}
\label{bound-final}
  d_R\Bigl(e^{i (t + \theta(t))} \psi(\cdot + \xi(t),t)\,,u_0\Bigr)
  \,\le\, \epsilon.
\end{equation}
Moreover $\xi$ and $\theta$ are continuously differentiable
functions of $t$ which satisfy
\begin{equation}
\label{bound-time-per}
  |\dot \xi(t)| + |\dot \theta(t)| \,\le\, C \epsilon, \quad t \in \R,
\end{equation}
for some positive constant $C$.
\end{theorem}

\begin{remark}
It is known from the work of Zhidkov \cite{Zhidkov} that the Cauchy
problem for the NLS equation \eqref{nls} is globally well-posed in
$X$. This is the functional framework that is used to define
solutions of \eqref{nls} in Theorem~\ref{theorem-soliton}.
\end{remark}

\begin{remark}
Except for the use of a different distance $d_R$, which controls the
perturbations in the topology of $H^2_{\rm loc}(\R)$,
Theorem~\ref{theorem-soliton} is the exact analogue of the orbital
stability results obtained in \cite{BGSS,GS}. However the proof is
quite different, and in some sense simpler, because the profile $u_0$
of the black soliton is an unconstrained local minimizer of the
higher-order functional $\Lambda$.
\end{remark}

\begin{remark}
It is also possible to prove asymptotic stability results
for the black soliton of the cubic NLS equation \eqref{nls}. In that
perspective, it is useful to consider the black soliton as
a member of the one-parameter family of traveling dark solitons,
given by the exact expression
\begin{equation}
\label{darksoliton}
  e^{it}\psi_{\nu}(x + \nu t,t) \,=\, \sqrt{{\TS 1 - \frac{1}{2} \nu^2}}
  \,\tanh\left(\sqrt{{\TS\frac{1}{2} - \frac{1}{4} \nu^2}}\,x\right)
  + \frac{i\nu}{\sqrt{2}},
\end{equation}
where $\nu \in (-\sqrt{2},\sqrt{2})$. Asymptotic stability of the
family of dark solitons with nonzero speed $\nu$ was proved in
\cite{Bethuel}, using the Madelung transformation and the hydrodynamic
formulation of the NLS equation. This approach applies to solutions
whose modulus is strictly positive, and therefore excludes the case of
the black soliton. Very recently, the asymptotic stability of the
black soliton (within the one-parameter family of all dark solitons)
has been established in \cite{Cuccagna-private,GS}. 
\end{remark}

The rest of this article is organized as follows. In
Section~\ref{sec:positive} we establish positivity and coercivity
properties for the quadratic form associated with the second variation
of the functional \eqref{Lamdef} at $u_0$. In Section~\ref{sec:modulation},
we introduce modulation parameters in a neighborhood of the soliton
profile to eliminate the zero modes of the second variation
$\Lambda''(u_0)$. Combining these results and using a new variable
borrowed from \cite{GS}, we prove in Section~\ref{sec:stability} the
orbital stability of the black soliton \eqref{wave} in the space $X$.

\section{Positivity and coercivity of the second variation}
\label{sec:positive}

Let $u_0$ be the soliton profile \eqref{black-soliton} and $\Lambda =
S - 2E$ be the functional defined by 
\eqref{energy}, \eqref{Sdef}, and \eqref{Lamdef}. In this section, 
we prove that the second variation
$\Lambda''(u_0)$ is nonnegative, as stated in
Proposition~\ref{prop-main}, and we deduce some coercivity properties
that will be used in the proof of Theorem~\ref{theorem-soliton}.  We
consider perturbations of $u_0$ of the form $\psi = u_0 + u + i v$,
where $u,v \in H^2(\R)$ are real-valued. As in \cite{GP}, the second
variations at $u_0$ of the functionals $E$ and $S$ satisfy
\begin{align*}
  \half \langle E''(u_0)[u,v], [u,v]\rangle \,&=\,
  \langle L_+ u,u\rangle_{L^2} + \langle L_- v,v\rangle_{L^2}, \\[1mm]
  \half \langle S''(u_0)[u,v], [u,v]\rangle \,&=\,
  \langle M_+ u,u\rangle_{L^2} + \langle M_- v,v\rangle_{L^2},
\end{align*}
where $\langle\cdot\,,\cdot\rangle_{L^2}$ denotes the usual scalar
product in $L^2(\R)$. The self-adjoint operators $L_\pm$ and $M_\pm$
have the following expressions:
\begin{equation}
\label{operatorsdef}
\begin{array}{l}
  L_+ \,=\, -\partial_x^2 + 3 u_0^2 - 1, \\[1mm]
  L_- \,=\, -\partial_x^2 + u_0^2 - 1,
\end{array} \qquad
\begin{array}{lcl}
  M_+ \,=\, \partial_x^4 - 5 \partial_x u_0^2 \partial_x -5 u_0^4 +
  15 u_0^2 - 4, \\[1mm]
  M_- \,=\, \partial_x^4 - 3 \partial_x u_0^2 \partial_x + u_0^2 - 1.
\end{array}
\end{equation}
In view of \eqref{Lamdef}, it follows that
\begin{equation}
\label{Lambdasecond}
  \half \langle \Lambda''(u_0)[u,v], [u,v]\rangle \,=\,
  \langle K_+ u,u\rangle_{L^2} + \langle K_- v,v\rangle_{L^2},
\end{equation}
where $K_{\pm} = M_{\pm} - 2L_{\pm}$. More explicitly, the quadratic
forms associated with $K_\pm$ are given by
\begin{align}
\label{Kquad+}
  \langle K_+ u,u\rangle_{L^2} \,&=\, \int_\R \Bigl(u_{xx}^2 +
  (5u_0^2-2)u_x^2 + (9u_0^2 -5u_0^4-2)u^2\Bigr)\dd x, \\ \label{Kquad-}
  \langle K_- v,v\rangle_{L^2} \,&=\, \int_\R \Bigl(v_{xx}^2 +
  (3u_0^2-2)v_x^2 + (1-u_0^2)v^2\Bigr)\dd x.
\end{align}

Our first task is to show that the quadratic forms \eqref{Kquad+},
\eqref{Kquad-} are nonnegative on $H^2(\R)$. Due to translation
invariance of the NLS equation \eqref{nls}, we have $L_+ u_0' = M_+
u_0' = 0$, hence also $K_+ u_0' = 0$. As $u_0' \in H^2(\R)$, this
shows that the quadratic form associated with $K_+$ has a neutral
direction, hence is not strictly positive, see Lemma
\ref{lemma-K-plus} below.  The situation is slightly different for
$K_-$: due to gauge invariance, we have $L_- u_0 = M_- u_0 = 0$, hence
$K_- u_0 = 0$, but of course $u_0 \not\in H^2(\R)$. In fact, the
result of Lemma \ref{lemma-K-minus} below shows that the quadratic
form associated with $K_-$ is strictly positive on $H^2(\R)$.

\medskip
We first prove that the quadratic form \eqref{Kquad+} is
nonnegative, see also \cite[Corollary~4.5]{GP}.

\begin{lemma}
\label{lemma-K-plus}
For any $u \in H^2(\R)$, we have
\begin{equation}
\label{operator-K-plus-soliton}
  \langle K_+ u, u \rangle_{L^2} \,=\, \|w_x\|_{L^2}^2 + \|w\|_{L^2}^2
  \,\ge\, 0,
\end{equation}
where $w = u_x + \sqrt{2} u_0 u$.
\end{lemma}

\begin{proof}
Integrating by parts and using the differential equations
\eqref{wave} satisfied by $u_0$, we easily obtain
\begin{equation}\label{LemK+1}
  \int_\R w^2 \dd x \,=\, \int_\R \Bigl(u_x^2 + 2\sqrt{2}u_0
  u u_x + 2u_0^2 u^2\Bigr) \dd x \,=\, \int_\R
  \Bigl(u_x^2 + (3u_0^2-1)u^2\Bigr)\dd x.
\end{equation}
Similarly, as $w_x = u_{xx} + \sqrt{2}u_0u_x + \sqrt{2}u_0'u$, we
find
\begin{align}\nonumber
  \int_\R w_x^2 \dd x \,&=\,  \int_\R \Bigl(u_{xx}^2 +  2\sqrt{2}u_0
  u_x u_{xx}+ 2u_0^2 u_x^2 + 2\sqrt{2}u_0' uu_{xx} + 4u_0 u_0' uu_x
  + 2u_0'^2 u^2\Bigr)\dd x \\ \nonumber
  \,&=\, \int_\R \Bigl(u_{xx}^2 + (5u_0^2-3)u_x^2 + 8u_0 u_0' uu_x
  + 2u_0'^2 u^2\Bigr)\dd x \\ \label{LemK+2}
  \,&=\, \int_\R \Bigl(u_{xx}^2 + (5u_0^2-3)u_x^2 + (1-u_0^2)(5u_0^2-1)
  u^2\Bigr)\dd x,
\end{align}
because $2u_0'^2 - 4(u_0u_0')' = (1-u_0^2)(5u_0^2-1)$. If we now
combine \eqref{LemK+1} and \eqref{LemK+2}, we see that
$\|w_x\|_{L^2}^2 + \|w\|_{L^2}^2$ is equal to the right-hand side of
\eqref{Kquad+}, which is the desired conclusion.
\end{proof}

\begin{remark}
\label{remark-black-plus}
The right-hand side of \eqref{operator-K-plus-soliton}
vanishes if and only if $w = 0$, which is equivalent to
$u = C u_0'$ for some constant $C$. Thus zero is a simple
eigenvalue of $K_+$ in $L^2(\R)$. Moreover, since $u_0(x)
\to \pm 1$ as $x \to \pm \infty$, it is clear from \eqref{Kquad+}
that the essential spectrum of $K_+$ is the interval $[2,\infty)$.
Thus if we restrict ourselves to the orthogonal complement of
$u_0'$ with respect to the scalar product $\langle\cdot\,,
\cdot\rangle_{L^2}$, the spectrum of $K_+$ is bounded from below
by a strictly positive constant, and the corresponding quadratic
form is thus coercive in the topology of $H^2(\R)$, see
Remark~\ref{oldcond} below.
\end{remark}

We next prove the positivity of the quadratic form \eqref{Kquad-},
see also \cite[Lemma~4.1]{GP}.

\begin{lemma}
\label{lemma-K-minus}
For any $v \in H^2(\R)$, we have
\begin{equation}
\label{operator-K-minus}
  \langle K_- v, v \rangle_{L^2} \,=\, \| L_- v \|_{L^2}^2 +
  \| u_0 v_x - u_0' v \|_{L^2}^2 \,\ge\, 0,
\end{equation}
where $L_- = -\partial_x^2 + u_0^2 - 1$.
\end{lemma}

\begin{proof}
Integrating by parts we obtain
\begin{align*}
  \int_\R (L_-v)^2 \dd x \,&=\, \int_{\R} \Bigl(v_{xx}^2 + 2(1 - u_0^2) v v_{xx}
  + (1-u_0^2)^2 v^2 \Bigr) \dd x \\ \,&=\, \int_{\R} \Bigl(v_{xx}^2
  + 2(u_0^2-1) v_x^2 -2(u_0 u_0')'v^2 + (1-u_0^2)^2v^2 \Bigr) \dd x.
\end{align*}
Similarly, we have
\[
 \int_\R \Bigl(u_0 v_x - u_0' v\Bigr)^2\dd x \,=\,
 \int_{\R} \Bigl(u_0^2 v_x^2 + (u_0u_0')'v^2 + u_0'^2 v^2\Bigr) \dd x.
\]
It follows that
\[
  \|L_- v\|_{L^2}^2 + \|u_0 v_x - u_0' v\|_{L^2}^2 \,=\, \int_{\R}
  \Bigl(v_{xx}^2 + (3 u_0^2 - 2) v_x^2 + [(1-u_0^2)^2 - u_0 u_0''] v^2
  \Bigl)\dd x,
\]
and that expression coincides with the right-hand side of
\eqref{Kquad-} since $(1-u_0^2)^2 - u_0 u_0'' = 1 - u_0^2$ by
\eqref{wave}. This proves \eqref{operator-K-minus}.
\end{proof}

\begin{remark}
The right-hand side of \eqref{operator-K-minus} vanishes if and only
if $L_-v = 0$ and $u_0 v_x - u_0' v = 0$, namely if $v = C u_0$ for
some constant $C$. As $u_0 \notin H^2(\R)$, this shows that $\langle K_-
v, v \rangle_{L^2} > 0$ for any nonzero $v \in H^2(\R)$. However,
since $|u_0(x)| \to 1$ as $|x| \to \infty$, it is clear from the
representation \eqref{Kquad-} that zero belongs to the essential
spectrum of the operator $K_-$, hence the associated quadratic form is
not coercive in the topology of $H^2(\R)$. Some weaker coercivity
property will nevertheless be established below, see
Remark~\ref{newcond}.
\end{remark}

\begin{remark}
In view of the decomposition \eqref{Lambdasecond},
Proposition~\ref{prop-main} is an immediate consequence of
Lemmas~\ref{lemma-K-plus} and \ref{lemma-K-minus}.
\end{remark}

In the rest of this section, we show that the quadratic
forms \eqref{Kquad+}, \eqref{Kquad-} are not only positive,
but also coercive in some appropriate sense.

\begin{lemma}
\label{lemma-soliton-1}
Let $u_0$ be the black soliton \eqref{black-soliton}. There exists a
positive constant $C$ such that, for any $u \in H^2(\R)$
satisfying $\langle u_0', u \rangle_{L^2} = 0$, we have the estimate
\begin{equation}
\label{bound-u}
  \| u \|_{H^2} \,\le\, C \| w \|_{H^1},
\end{equation}
where $w = u_x + \sqrt{2} u_0 u$.
\end{lemma}

\begin{proof}
Solving the linear differential equation $u_x + \sqrt{2} u_0 u = w$
by Duhamel's formula, we find $u = A u_0' + W$ for some $A \in \R$,
where
\begin{equation}
\label{variation-u}
  W(x) \,=\, \int_0^x K(x,y) w(y) \dd y, \qquad K(x,y) \,=\,
 \frac{\cosh^2(y/\sqrt{2})}{\cosh^2(x/\sqrt{2})}.
\end{equation}
The constant $A$ is uniquely determined by the orthogonality
condition $\langle u_0', u \rangle_{L^2} = 0$, which implies that
$A \|u_0'\|_{L^2}^2 + \langle u_0', W \rangle_{L^2} = 0$. Using
\eqref{variation-u}, we easily obtain
\begin{align}\nonumber
  \langle u_0', W \rangle_{L^2} \,&=\, \int_{-\infty}^\infty \biggl\{\int_0^x
  K(x,y) w(y) \dd y\biggr\}u_0'(x) \dd x \\ \nonumber
  \,&=\, \int_0^\infty \biggl\{\int_y^\infty
  K(x,y) u_0'(x)\dd x\biggr\}\Bigl(w(y)-w(-y)\Bigr) \dd y \\ \label{intex}
  \,&=\,  \frac{1}{3} \int_0^\infty e^{-\sqrt{2}y}\,\frac{3+e^{-\sqrt{2}y}}{
  1+e^{-\sqrt{2}y}}\Bigl(w(y)-w(-y)\Bigr) \dd y,
\end{align}
hence $|\langle u_0', W \rangle_{L^2}| \le 2^{-1/4}\|w\|_{L^2}$.
It follows that $|A| \le C\|w\|_{L^2}$ for some $C > 0$.

On the other hand, if we introduce the operator notation $W =
\hat{K}(w)$ for the representation \eqref{variation-u}, we note
that $\hat{K}$ is a bounded operator from $L^{\infty}(\R)$ to
$L^{\infty}(\R)$ with norm
$$
  K_{\infty} \,=\, \sup_{x \in \R} \int_0^{|x|} K(x,y) \dd y
  \,=\, \frac{1}{\sqrt{2}}
  \sup_{x \in \R} \frac{1 + 2 \sqrt{2} |x| e^{-\sqrt{2}|x|}
  - e^{-2 \sqrt{2}|x|}}{1 + 2 e^{-\sqrt{2}|x|} + e^{-2 \sqrt{2}|x|}}
  \,<\, \infty,
$$
as well as a bounded operator from $L^1(\R)$ to
$L^1(\R)$ with norm
$$
  K_1 \,=\, \sup_{y \in \R} \int_{|y|}^{\infty} K(x,y) \dd x
  \,=\, \frac{1}{\sqrt{2}}
  \sup_{y \in \R} \Bigl(1 + e^{-\sqrt{2}|y|}\Bigr) = \sqrt{2}.
$$
By the Riesz-Thorin interpolation theorem, it follows that $\hat{K}$
is a bounded operator from $L^2(\R)$ to $L^2(\R)$, and we have the
estimate $\|W\|_{L^2} = \|\hat{K}(w) \|_{L^2} \leq (K_1 K_{\infty})^{1/2}
\|w\|_{L^2}$.

Summarizing, we have shown that $\|u\|_{L^2} \le |A| \|u_0'\|_{L^2} +
\|W\|_{L^2} \le C \|w\|_{L^2}$ for some $C > 0$. Since $w = u_x +
\sqrt{2} u_0 u$, we also have $\|u_x\|_{L^2} \le \|w\|_{L^2} +
\sqrt{2}\|u\|_{L^2}$ and (after differentiating) $\|u_{xx}\|_{L^2} \le
\|w_x\|_{L^2} + \sqrt{2} \|u_x\|_{L^2} + \| u \|_{L^2}$. This proves
the bound \eqref{bound-u}.
\end{proof}

\begin{remark}\label{oldcond}
Combining \eqref{operator-K-plus-soliton} and \eqref{bound-u}, we conclude
that there exists a constant $C_+ > 0$ such that
\begin{equation}
\label{K+coercive}
 \langle K_+ u, u \rangle_{L^2} \,\ge\, C_+ \|u\|_{H^2}^2,
\end{equation}
for all $u \in H^2(\R)$ satisfying $\langle u_0', u \rangle_{L^2} = 0$.
\end{remark}

\begin{lemma}
\label{lemma-soliton-2}
Let $u_0$ be the black soliton \eqref{black-soliton}. There exists
a positive constant $C$ such that, for any $v \in
H^2_{\rm loc}(\R)$ satisfying $v_x \in H^1(\R)$ and
$\langle u_0'', v \rangle_{L^2} = 0$, we have the estimate
\begin{equation}
\label{bound-v}
  \| v_{xx} \|_{L^2} + \| v_x \|_{L^2} + |v(0)|  \,\le\, C(\| p \|_{L^2}
  + \| q \|_{L^2}),
\end{equation}
where $p = u_0 v_x - u_0' v$ and $q = -L_- v = v_{xx} + (1-u_0^2) v$.
\end{lemma}

\begin{proof}
Any solution of the linear differential equation $u_0 v_x - u_0' v = p$
has the form
$v = B u_0 + Z$ for some $B \in \R$, where
\begin{equation}
\label{variation-v}
  Z(x) \,=\, u_0(x)\int_0^x \Bigl(p(y) + \sqrt{2}q(y)\Bigr)\dd y
  - \sqrt{2}p(x).
\end{equation}
Indeed, we observe that $p_x = u_0 v_{xx} - u_0'' v = u_0 (v_{xx}
+ (1-u_0^2) v) = u_0 q$. Thus, if $v = B u_0 + Z$, we have
\begin{equation}
\label{variation-v2}
  v_x(x) \,=\, u_0'(x)\left(B + \int_0^x \Bigl(p(y) + \sqrt{2}q(y)
  \Bigr)\dd y\right) + u_0(x)p(x),
\end{equation}
hence $u_0 v_x - u_0' v = (u_0^2 + \sqrt{2}u_0')p = p$. The constant
$B$ is uniquely determined by the orthogonality condition
$\langle u_0'', v \rangle_{L^2} = 0$, which implies that
$B \|u_0'\|_{L^2}^2 = \langle u_0'',Z\rangle$.

Since $p \in L^2(\R)$ and $p_x = u_0 q \in L^2(\R)$, we have $p \in
L^\infty(\R)$ by Sobolev's embedding resulting in the bound
$\|p\|_{L^\infty}^2 \le \|p\|_{L^2} \|p_x\|_{L^2} \le
\|p\|_{L^2}\|q\|_{L^2}$. Thus, using \eqref{variation-v} and
H\"older's inequality, we deduce that
$$
  |Z(x)| \,\le\, \sqrt{2} (|x|^{1/2} + 1)(\|p\|_{L^2} + \|q\|_{L^2}), \quad
  x \in \R.
$$
This moderate growth of $Z$ is compensated for by the exponential
decay of $u_0''$ to zero at infinity, and we obtain $|\langle u_0'',Z\rangle|
\le C(\|p\|_{L^2} + \|q\|_{L^2})$ for some $C > 0$, hence also
$|B| \le C(\|p\|_{L^2} + \|q\|_{L^2})$. In the same way, it follows
from \eqref{variation-v2} that $\|v_x\|_{L^2} \le C(\|p\|_{L^2} +
\|q\|_{L^2})$. A similar estimate holds for $\|v_{xx}\|_{L^2}$
because $v_{xx} = q - (1 - u_0^2) v$ and $1 - u_0^2$ has the exponential
decay to zero at infinity. Finally, since $v(0) =
-\sqrt{2}p(0)$, we also have $|v(0)| \le C(\|p\|_{L^2} + \|q\|_{L^2})$.
This proves the bound \eqref{bound-v}.
\end{proof}

\begin{remark}
\label{newcond}
Combining \eqref{operator-K-minus}
and \eqref{bound-v}, we conclude that there exists a constant
$C_- > 0$ such that
\begin{equation}
\label{K-coercive}
  \langle K_- v, v \rangle_{L^2} \,\ge\, C_- \Bigl(\|v_x\|_{H^1}^2
  + |v(0)|^2\Bigr),
\end{equation}
for all $v \in H^2_{\rm loc}(\R)$ satisfying $v_x \in H^1(\R)$ and
$\langle u_0'', v \rangle_{L^2} = 0$. As is clear from the
proof of Lemma~\ref{lemma-soliton-2}, we need some orthogonality
condition on $v$ to prove estimate \eqref{K-coercive}, and
since $u_0 \notin L^2(\R)$ we cannot impose $\langle u_0,
v \rangle_{L^2} = 0$. Thus we use $u_0'' = u_0(u_0^2-1)$ instead
of $u_0$. Although $u_0''$ is only an approximate eigenfunction
of $K_-$, the orthogonality condition $\langle u_0'', v \rangle_{L^2} = 0$ is
good enough for our purposes, as we shall see in Section~\ref{sec:modulation}.
\end{remark}

\section{Modulation parameters near the black soliton}
\label{sec:modulation}

This section contains some important preliminary steps in the
proof of Theorem \ref{theorem-soliton}. To establish the orbital
stability of the black soliton with profile $u_0$, our
general strategy is to consider solutions $\psi(x,t)$ of the
cubic NLS equation \eqref{nls} of the form
\begin{equation}
\label{decomposition2}
  e^{i(t + \theta(t))} \psi(x + \xi(t),t) \,=\, u_0(x) + u(x,t) + i v(x,t),
  \qquad (x,t) \in \R \times \R,
\end{equation}
where the perturbations $u,v$ are real-valued and satisfy the
orthogonality conditions
\begin{equation}
\label{projections2}
  \langle u_0', u(\cdot,t) \rangle_{L^2} \,=\, 0, \qquad
  \langle u_0'', v(\cdot,t) \rangle_{L^2} \,=\, 0, \qquad t \in \R.
\end{equation}
As was discussed in Remarks \ref{oldcond} and \ref{newcond}, these
conditions are needed to exploit the coercivity properties of the
second variation $\Lambda''(u_0)$, where $\Lambda$ is the conserved
quantity \eqref{Lamdef}. They also allow us to determine uniquely the
``modulation parameters'', namely the translation $\xi(t)$ and the
phase $\theta(t)$, at least for solutions $\psi(x,t)$ in a small
neighborhood of the black soliton. To make these considerations
rigorous, we first need to specify in which topology that neighborhood
is understood; in other words, we need to choose an appropriate
perturbation space. Next we have to verify that the modulation
parameters exist and depend smoothly on the solution $\psi(x,t)$ in
the vicinity of the black soliton.

Concerning the first point, we observe that the functional
\eqref{Lamdef} which serves as a basis for our analysis is invariant
under translations and gauge transformations, and we recall that
$\Lambda'(u_0) = 0$. Thus, if $\psi(x,t)$ is a solution of the NLS
equation \eqref{nls} of the form \eqref{decomposition2} with $u,v \in
H^2(\R)$, we have for each fixed $t \in \R$ the following expansion
\begin{equation}
\label{DeltaLambda2}
  \Lambda(\psi) - \Lambda(u_0) \,=\, \langle K_+ u, u \rangle_{L^2}
  + \langle K_- v, v \rangle_{L^2} + N(u,v),
\end{equation}
where $N(u,v)$ collects all terms that are at least cubic in $u$ and
$v$. However, unlike in the periodic case considered in the companion
paper \cite{GP}, the decomposition \eqref{DeltaLambda2} is not
sufficient to prove the orbital stability of the black
soliton. Indeed, the quadratic terms in \eqref{DeltaLambda2} are
nonnegative, but they are degenerate in the sense that they do not
control the $L^2(\mathbb{R})$ norm of $v$, as can be seen from the
lower bound \eqref{K-coercive}. This is due to the fact that the
operator $K_-$ has essential spectrum touching the origin, with
generalized eigenfunctions corresponding to slow modulations of the
phase of the black soliton. As is clear from the proof of
Lemma~\ref{lemma-soliton-2}, one cannot even prove that $v \in
L^\infty(\R)$ if we only know that $\langle K_- v, v \rangle_{L^2} <
\infty$. This in turn makes it impossible to control the nonlinearity
$N(u,v)$ in \eqref{DeltaLambda2} in terms of the quadratic part
$\langle K_+ u, u \rangle_{L^2} + \langle K_- v, v \rangle_{L^2}$.

There are good reasons to believe that the above problem is not just a
technical one, and that the $H^2$ topology for the perturbations $u,v$
is not appropriate to prove orbital stability of the black soliton.
Indeed, as is well known, the cubic NLS equation \eqref{nls} has a
family of travelling dark solitons $\psi_{\nu}(x,t)$ given by
\eqref{darksoliton}.  Rigorous results \cite{GS} and numerical
simulations indicate that a small, localized perturbation of the
black soliton $\psi_0$ can lead to the formation of a dark soliton
$\psi_{\nu}$ with a small nonzero speed $\nu$. If this happens, the
functions $u,v$ defined in \eqref{decomposition2} cannot stay bounded
in $L^2(\R)$ for all times, because $\psi_{\nu} - \psi_0 \notin
L^2(\R)$ if $\nu \neq 0$. Note, however, that the quantity
$|\psi_{\nu}| - |\psi_0|$ does belong to $L^2(\R)$ and decays
exponentially at infinity. This suggests that a particular combination
of $u,v$ may be controlled in $L^2(\R)$ for all times.

Following \cite{GS}, we introduce the auxiliary variable
\begin{equation}
\label{etadef}
  \eta \,=\, |u_0 + u + iv|^2 - |u_0|^2 \,=\, 2 u_0 u + u^2 + v^2,
\end{equation}
which allows us to control the perturbations of the modulus of the
black soliton $u_0$. The idea is now to consider perturbations $u,v$
for which $u_x, v_x \in H^1(\R)$, $\eta \in L^2(\R)$, and $u,v \in
L^2(-R,R)$ for some fixed $R \ge 1$. If $\psi = u_0 + u + iv$, this is
equivalent to requiring that $\psi \in X$, where $X$ is the function
space \eqref{Xdef}, or that $d_R(\psi,u_0) < \infty$, where
$d_R$ is the distance \eqref{distance}. Indeed, we have by definition
\begin{equation}
\label{distance2}
  d_R(\psi,u_0) \,=\, \|u_x + iv_x\|_{H^1(\R)} + \|\eta\|_{L^2(\R)}
  + \|u + iv\|_{L^2(-R,R)}.
\end{equation}
Note, however, that we do not assume any longer that $u,v$ are
square integrable at infinity. In particular, the perturbed
solutions we consider include dark solitons $\psi_{\nu}$
with nonzero speed $\nu$.

Now that we have defined a precise perturbation space, we can state
our first result showing the existence and the continuity of the
modulation parameters $\xi$ and $\theta$ in a neighborhood of the
orbit of the soliton profile $u_0$. The following statement is
very close in spirit to Proposition~2 in \cite{GS} or
Lemma~6.1 in \cite{GP}.

\begin{lemma}
\label{lemma-xith}
Fix any $R \ge 1$. There exists $\epsilon_0 > 0$ such that,
for any $\psi \in X$ satisfying
\begin{equation}
\label{inf2}
  \inf_{\xi, \theta \in \R} d_R\Bigl(e^{i \theta} \psi(\cdot + \xi),
  u_0\Bigr) \,\le\, \epsilon_0,
\end{equation}
there exist $\xi \in \R$ and $\theta \in \R/(2\pi\Z)$ such that
\begin{equation}
\label{decomp2}
  e^{i \theta} \psi(x + \xi) \,=\, u_0(x) + u(x) + i v(x), \quad x \in \R,
\end{equation}
where the real-valued functions $u$ and $v$ satisfy the orthogonality
conditions \eqref{projections2}. Moreover, the modulation parameters
$\xi \in \R$ and $\theta \in \R/(2\pi\Z)$ depend continuously on
$\psi$ in the topology defined by the distance \eqref{distance}.
\end{lemma}

\begin{proof}
It is sufficient to prove \eqref{decomp2} for all $\psi \in X$
such that $\epsilon := d_R(\psi,u_0)$ is sufficiently small. Given
such a $\psi \in X$, we consider the smooth function ${\bf f} : \R^2
\to \R^2$ defined by
\[
  {\bf f}(\xi,\theta) \,=\, \begin{pmatrix} \langle u_0'(\cdot
  - \xi), {\rm Re}(e^{i \theta} \psi) \rangle_{L^2} \\[1mm]
  \langle u_0''(\cdot - \xi), {\rm Im}(e^{i \theta} \psi)
  \rangle_{L^2} \end{pmatrix}, \qquad (\xi,\theta) \in \R^2.
\]
By construction, we have ${\bf f}(\xi,\theta) = {\bf 0}$ if and only if
$\psi$ can be represented as in \eqref{decomp2} for some
real-valued functions $u,v$ satisfying the orthogonality conditions
\eqref{projections2}.

If we decompose $\psi = u_0 + u + iv$ where $u,v$ are real-valued,
we have $\langle u_0', {\rm Re}(\psi)\rangle_{L^2} = \langle u_0',
u\rangle_{L^2}$ because $\langle u_0',u_0\rangle_{L^2} = 0$. As in
the proof of Lemma~\ref{lemma-soliton-2}, we observe that
\[
  |u(x)| \,\le\, C\Bigl(\|u\|_{L^2(-1,1)} + (1+|x|^{1/2})\|u_x\|_{L^2(\R)}
  \Bigr) \,\le\, C(1+|x|^{1/2})d_R(\psi,u_0),
\]
where in the last inequality we have used \eqref{distance2}.
Thus $|\langle u_0', {\rm Re}(\psi)\rangle_{L^2}| \le C d_R(\psi,u_0)$,
and a similar argument gives $|\langle u_0'', {\rm Im}(\psi)
\rangle_{L^2}| \le C d_R(\psi,u_0)$. This shows that $\|{\bf f}
(0,0) \| \le C \epsilon$ for some positive constant $C$
independent of $\epsilon$.

On the other hand, the Jacobian matrix of the function ${\bf f}$
at the origin $(0,0)$ is given by
\[
  D {\bf f}(0,0) \,=\,
  \begin{pmatrix} \|u_0'\|_{L^2}^2 & 0 \\ 0 & -\|u_0'\|_{L^2}^2
  \end{pmatrix} \,+\, \begin{pmatrix} -\langle u_0'',{\rm Re}(\psi - u_0)
  \rangle_{L^2} & -\langle u_0', {\rm Im}(\psi - u_0)\rangle_{L^2} \\[.5mm]
  -\langle u_0''', {\rm Im}(\psi- u_0)\rangle_{L^2} &
  \langle u_0'', {\rm Re}(\psi - u_0)\rangle_{L^2} \end{pmatrix}.
\]
The first term in the right-hand side is a fixed invertible matrix and
the second term is bounded in norm by $C\epsilon$, hence $D {\bf
  f}(0,0)$ is invertible if $\epsilon$ is small enough.  In addition,
the norm of the inverse of $D {\bf f}(0,0)$ is bounded by a constant
independent of $\epsilon$. Finally, it is straightforward to verify
that the second-order derivatives of ${\bf f}$ are uniformly bounded
when $\epsilon \le 1$. These observations together imply that there
exists a unique pair $(\xi,\theta)$, in a neighborhood of size
$\mathcal{O}(\epsilon)$ of the origin, such that ${\bf f}(\xi,\theta)
= {\bf 0}$. Thus the decomposition \eqref{decomposition2} holds for
these values of $(\xi,\theta)$. In addition, the above argument shows 
that the modulation parameters $\xi,\theta$ depend continuously on
$\psi \in X$ in the topology defined by the distance \eqref{distance}.
This concludes the proof.
\end{proof}

As was already mentioned, the Cauchy problem for the NLS equation
\eqref{nls} is globally well-posed in the space $X$ \cite{Zhidkov}.
If $\psi(\cdot,t)$ is a solution of \eqref{nls} in $X$ which stays for
all times in a neighborhood of the orbit of the black soliton, the
modulation parameters $\xi(t)$, $\theta(t)$ given by
the decomposition (\ref{decomposition2}) subject to the orthogonality 
conditions (\ref{projections2}) are continuous functions of time.
In fact, as in \cite[Lemma 6.3]{GP}, we have the following
stronger conclusion:

\begin{lemma}
\label{difflem}
If $\epsilon > 0$ is sufficiently small, and if $\psi(\cdot,t)$ is
any solution of the NLS equation \eqref{nls} satisfying estimate
\eqref{bound-final} for all $t \in \R$, then the modulation parameters
$\xi(t),\theta(t)$ in the decomposition (\ref{decomposition2}) subject to 
(\ref{projections2}) are continuously
differentiable functions of $t$ satisfying \eqref{bound-time-per}.
\end{lemma}

\begin{proof}
If $\psi(\cdot,t)$ is any solution of the NLS equation \eqref{nls}
in $X$, we know from \cite{Gerard,Zhidkov} that $t \mapsto
\psi(\cdot,t)$ is continuous in the topology defined by the
distance \eqref{distance}. Thus, if estimate \eqref{bound-final}  holds
for all $t \in \mathbb{R}$, Lemma~\ref{lemma-xith} shows that $\psi(\cdot,t)$
can be decomposed as in \eqref{decomposition2} with modulation
parameters $\xi(t),\theta(t)$ that depend continuously on time.
To prove differentiability, we first consider more regular solutions
for which $\psi(\cdot,t) \in Y$, where
\[
  Y \,=\, \Bigl\{\psi \in H^4_{\rm loc}(\R)\,: \quad \psi_x \in
  H^3(\R), ~1 - |\psi|^2 \in L^2(\R) \Bigr\}.
\]
For such solutions, it is not difficult to verify (by inspecting the
proof of Lemma~\ref{lemma-xith}) that the modulation parameters are
$C^1$ functions of time, so that we can differentiate both sides of
\eqref{decomposition2} and obtain from (\ref{nls}) the evolution
system
\[
  \left\{\!\!\begin{array}{l}
  ~\,\,u_t \,=\, L_- v + \dot{\xi} (u_0' + u_x) - \dot{\theta} v +
  (2 u_0 u + u^2 + v^2) v, \\
  -v_t \,=\, L_+ u -\dot{\xi} v_x - \dot{\theta} (u_0 + u) +
  (3 u_0 u + u^2 + v^2) u + u_0 v^2, \end{array} \right.
\]
where the operators $L_\pm$ are defined in \eqref{operatorsdef}.
Using the orthogonality conditions \eqref{projections2}, we
eliminate the time derivatives $u_t, v_t$ by taking the scalar
product of the first line with $u_0'$ and of the second line with
$u_0''$. This gives the following linear system for the derivatives
$\dot{\xi}$ and $\dot{\theta}$:
\begin{equation}
\label{Bsys}
  B \begin{pmatrix} \dot{\xi} \\[.5mm] \dot{\theta} \end{pmatrix}
  \,=\, \begin{pmatrix} \langle L_- u_0', v \rangle_{L^2} \\[.5mm]
  \langle L_+ u_0'', u \rangle_{L^2} \end{pmatrix} \,+\,
\begin{pmatrix}
  \langle u_0', (2 u_0 u + u^2 + v^2) v \rangle_{L^2} \\[.5mm]
  \langle u_0'', (3 u_0 u + u^2 + v^2) u + u_0 v^2 \rangle_{L^2}
\end{pmatrix},
\end{equation}
where
\begin{equation}
\label{BBdef}
  B \,=\, \begin{pmatrix} -\| u_0' \|^2_{L^2} & 0 \\ 0 &
  -\|u_0'\|^2_{L^2}\end{pmatrix} \,+\,
  \begin{pmatrix} -\langle u_0', u_x \rangle_{L^2} & \langle u_0', v
  \rangle_{L^2} \\[.5mm] \langle u_0'', v_x \rangle_{L^2} &
  \langle u_0'', u\rangle_{L^2} \end{pmatrix}.
\end{equation}
As in the proof of Lemma~\ref{lemma-xith}, it is easy to verify using
\eqref{bound-final} that the second term in the right-hand side of
\eqref{BBdef} is bounded by $C \epsilon$ for some positive constant
$C$, hence the matrix $B$ is invertible if $\epsilon$ is small
enough. Inverting $B$ in \eqref{Bsys}, we obtain a formula for the
derivatives $\dot{\xi},\dot{\theta}$ in which the right-hand side
makes sense (and is a continuous function of time) for any solution
$\psi(\cdot,t) \in X$ of \eqref{nls} satisfying \eqref{bound-final}
for all times. Since $Y$ is dense in $X$, we conclude by a standard
approximation argument that the modulation parameters
$\xi(t),\theta(t)$ are $C^1$ functions of time in the general case,
and that their derivatives satisfy \eqref{Bsys}.  Finally, the first
term in the right-hand side of \eqref{Bsys} is of size
$\mathcal{O}(\epsilon)$, whereas the second term is
$\mathcal{O}(\epsilon^2)$, hence $|\dot{\xi}(t)| + |\dot{\theta}(t)|
\le C\epsilon$ for all $t \in \R$, where the positive constant $C$ is
independent of $t$. This concludes the proof.
\end{proof}

\section{Proof of orbital stability of the black soliton}
\label{sec:stability}

This final section is entirely devoted to the proof of
Theorem~\ref{theorem-soliton}. As in the previous section, we consider
solutions of the NLS equation \eqref{nls} of the form
\eqref{decomposition2}, where the real-valued perturbations $u,v$
satisfy the orthogonality conditions \eqref{projections2}. Our main
task is a detailed analysis of the functional \eqref{Lamdef} in a
neighborhood of the orbit of the soliton profile $u_0$. Instead of
using the straightforward decomposition \eqref{DeltaLambda2}, the main
idea is to express the difference $\Lambda(\psi) - \Lambda(u_0)$ in
terms of the variables $u$, $v$, and $\eta$, where $\eta$ is
defined in \eqref{etadef}.

\begin{lemma}
\label{lemma-soliton-3}
If $\psi = u_0 + u + iv$ satisfies $d_R(\psi,u_0) < \infty$, then
\begin{align}\nonumber
  \Lambda(\psi) - \Lambda(u_0)
  \,=\, \int_{\R} \Bigl(&u_{xx}^2 + v_{xx}^2 + (3u_0^2-2)
  (u_x^2 + v_x^2) + (1-u_0^2)(u^2+v^2) \\ \label{Lamexp}
  &-3(1-u_0^2)(1-3u_0^2)u^2 + \frac12 \eta_x^2 + \frac12(3u_0^2-2)
  \eta^2\\ \nonumber
  &+\frac12 \eta^3 + 3\eta(u_x^2 + v_x^2) + 6u_0'(u^2+v^2)u_x\Bigr)\dd x~.
\end{align}
\end{lemma}

\begin{proof}
We observe that $|\psi|^2 = u_0^2 + \eta$ and $\bar \psi \psi_x +
\psi \bar \psi_x = 2u_0 u_0' + \eta_x$. Thus, if
$$
  A(\psi) \,=\, |\psi_{xx}|^2 + |\psi_x|^2 (3 |\psi|^2 -2) + \frac{1}{2}
 (\bar{\psi} \psi_x + \psi \bar{\psi}_x)^2 + \frac{1}{2} |\psi|^2
 (1 - |\psi|^2)^2
$$
denotes the integrand in the functional $\Lambda = S - 2 E$, a direct
calculation shows that
\begin{align}\nonumber
  A(\psi) - A(u_0) \,=~ &\mathcal{L}(u,\eta) + 6\eta u_0' u_x
  + u_{xx}^2 + v_{xx}^2 + (3u_0^2-2)(u_x^2 + v_x^2) \\ \label{Aexp}
  &+ \frac12 \eta_x^2 + \frac12(3u_0^2-2)\eta^2 + \frac12 \eta^3
  + 3\eta(u_x^2 + v_x^2),
\end{align}
where $\mathcal{L}(u,\eta) = 2u_0''u_{xx} + 2(3u_0^2-2)u_0'u_x
+ 2u_0u_0'\eta_x + \eta(1-u_0^2)(2-3u_0^2)$. We now integrate
the right-hand side of \eqref{Aexp} over $x \in \R$, starting with
the terms $\mathcal{L}(u,\eta)$ which are linear in $u$ and $\eta$.
Using the identities $u_0'' + u_0 - u_0^3 = 0$ and $u_0'''' +
(1-3u_0^2)u_0'' -6u_0 u_0'^2 = 0$, we find
\begin{align*}
  2\int_\R \Bigl(u_0''u_{xx} + (3u_0^2-2)u_0'u_x\Bigr)\dd x \,&=\,
  2\int_\R \Bigl(u_0'''' - (3u_0^2-2)u_0'' - 6u_0 u_0'^2\Bigr)u\dd x \\
  \,&=\, 2\int_\R u_0'' u \dd x \,=\, -2\int_\R (1-u_0^2)u_0 u \dd x.
\end{align*}
Similarly, as $2(u_0 u_0')' = (1-u_0^2)(1-3u_0^2)$, we have
$$
  2\int_\R u_0 u_0' \eta_x \dd x = -2\int_R (u_0 u_0')'\eta \dd x =
  - \int_\R (1-u_0^2)(1-3u_0^2)\eta d x.
$$
We conclude that
\begin{equation}
\label{Lterms}
  \int_\R \mathcal{L}(u,\eta)\dd x \,=\, \int_\R (1-u_0^2)
  (\eta - 2u_0 u)\dd x \,=\, \int_\R (1-u_0^2)(u^2+v^2)\dd x.
\end{equation}
Note that \eqref{Lterms} is now quadratic in $u$ and $v$, which could
be expected since $u_0$ is a critical point of the functional
$\Lambda$.  We next consider the quadratic term $6\eta u_0' u_x$ in
\eqref{Aexp}, which has no definite sign. Using the representation
\eqref{etadef}, we find $6\eta u_0' u_x = 12 u_0 u_0' u u_x + 6 u_0'
(u^2+v^2)u_x$, and integrating by parts, we obtain
\begin{equation}
\label{Qterm}
  6\int_\R \eta u_0' u_x\dd x \,=\, -3\int_\R (1-u_0^2)(1-3u_0^2)u^2
  \dd x + 6 \int_\R u_0' (u^2+v^2)u_x \dd x.
\end{equation}
Now, combining \eqref{Aexp}, \eqref{Lterms}, and \eqref{Qterm},
we arrive at \eqref{Lamexp}.
\end{proof}

To simplify the notations, we define
\begin{align}\nonumber
  B_0(u) \,&=\, u_{xx}^2 + (5u_0^2-2)u_x^2 - (1-3u_0^2)u^2
    - (1-u_0^2)(1-5u_0^2)u^2\\ \label{Bdef}
  B_1(u) \,&=\, u_{xx}^2 + (3u_0^2-2)u_x^2 + (1-u_0^2)u^2
    - 3 (1-u_0^2)(1-3u_0^2)u^2\\ \nonumber
  B_2(v) \,&=\, v_{xx}^2 + (3u_0^2-2)v_x^2 + (1-u_0^2)v^2\\ \nonumber
  B_3(\eta) \,&=\, {\TS \frac12\eta_x^2 + \frac12(3u_0^2-2)\eta^2}.
\end{align}
The quadratic terms in the right-hand side of \eqref{Lamexp}
can be written in the compact form
\begin{equation}
\label{Qdef}
  Q(u,v,\eta) = \int_\R \Bigl(B_1(u) + B_2(v) + B_3(\eta)\Bigr)\dd x.
\end{equation}
We see that $Q(u,v,\eta)$ contains $\langle K_-v,v\rangle \equiv
\int_\R B_2(v)\dd x$, but not $\langle K_+u,u\rangle \equiv \int_\R
B_0(u)\dd x$. Instead, it only contains $\int_\R B_1(u)\dd x$ and
$\int_\R B_3(\eta)\dd x$. This discrepancy is due to that fact that
the variables $u$ and $\eta$ are not independent. As $\eta = 2 u_0 u +
u^2 + v^2$, the quantity $\int_\R B_3(\eta)\dd x$ also contains
quadratic terms in $u$ and $u_x$, which should be added to $\int_{\R}
B_1(u) dx$ to obtain $\int_{\R} B_0(u) dx$.

Due to the relation between $u$ and $\eta$, it is not obvious that
each quadratic term in \eqref{Qdef} is positive independently of the
others. To avoid that difficulty, we fix some $R \ge 1$ (which will be
chosen large enough below) and we split the integration domain into
two regions. When $|x| \le R$, we replace $\eta$ by $2 u_0 u + u^2 +
v^2$, and we use extensions of Lemmas \ref{lemma-soliton-1} and
\ref{lemma-soliton-2} to prove positivity of the quadratic terms in
\eqref{Qdef}. In the outer region $|x| > R$, the analysis is much
simpler, because the expressions $B_1(u)$, $B_2(v)$, and $B_3(\eta)$
are obviously positive if $R$ is large enough.

Since $\eta$ is a nonlinear function of $u$ and $v$, the analysis
of the quadratic expression \eqref{Qdef} will produce higher-order
terms, which will be controlled using a smallness assumption on
the distance $d_R(\psi,u_0)$. To that purpose, we find it convenient
to introduce the quantity
\begin{equation}
\label{rhodef}
  \rho^2(u,v,\eta) \,=\, \int_\R \Bigl(u_{xx}^2 + v_{xx}^2 + u_x^2 +
  v_x^2 \Bigr)\dd x + \int_{|x|\le R} \Bigl(u^2 + R^{-2}v^2\Bigr)
  \dd x + \int_{|x|\ge R} \Bigl(\eta_x^2 + \eta^2\Bigr)\dd x,
\end{equation}
which is equivalent to the squared distance \eqref{distance2} in a
neighborhood of $u_0$. Indeed, we have the following elementary
result:

\begin{lemma}
\label{auxlem}
Fix $R \ge 1$, and assume that $\psi = u_0 + u + iv$, where
$u,v \in H^2_{\rm loc}(\R)$ are real-valued.
Let $d_R(\psi,u_0)$ be given by \eqref{distance2} and
$\rho(u,v,\eta)$ by \eqref{rhodef}.
\\[1mm]
{\bf a)} One has  $d_R(\psi,u_0) < \infty$ if and only if
$\rho(u,v,\eta) < \infty$.\\[1mm]
{\bf b)} There exists a constant $C_0 \ge 1$ (independent of $R$)
such that, if $d_R(\psi,u_0) \le 1$ or if\\
\null\hspace{6mm}$R^{1/2}\rho(u,v,\eta) \le 1$, then
\begin{equation}
\label{rhoequiv}
  C_0^{-1} \rho(u,v,\eta) \le d_R(\psi,u_0) \le C_0 R \rho(u,v,\eta).
\end{equation}
\end{lemma}

\begin{proof}
Throughout the proof, we denote $d_R(\psi,u_0)$ by $d_R$ and
$\rho(u,v,\eta)$ simply by $\rho$. We proceed in three steps.

\smallskip
\noindent{\bf Step 1:} Assume first that $d_R < \infty$, so that
$u_x,v_x \in H^1(\R)$, $u,v \in L^2(-R,R)$, and $\eta \in L^2(\R)$,
where $\eta = |\psi|^2 - |u_0|^2 = 2 u_0 u + u^2 + v^2$. We claim
that $u,v \in L^\infty(\R)$ and that
\begin{equation}
\label{Kbound1}
  K \,:=\, \|u\|_{L^\infty(\R)} + \|v\|_{L^\infty(\R)} \,\le\, C(1 + d_R),
\end{equation}
for some universal constant $C > 0$. Indeed, if $f = |\psi| - |u_0|$,
we observe that
$$
  d_R^2 \,\ge\, \int_\R \eta^2 \dd x \,\ge\, \int_{|x|\ge 1}
  (|\psi| - |u_0|)^2(|\psi| + |u_0|)^2\dd x \ge C \int_{|x|\ge 1}
  f^2 \dd x,
$$
hence $f \in L^2(I)$, where $I = \{x \in \R : |x| \ge 1\}$, and
$\|f\|_{L^2(I)} \le C d_R$. Moreover, we have $|f_x| \le
2 u_0' + |u_x| + |v_x|$ almost everywhere, hence $f_x \in L^2(\R)$
and $\|f_x\|_{L^2(\R)} \le C (1+d_R)$. By Sobolev embedding,
this implies that $f \in L^\infty(I)$, hence also $u,v \in L^\infty(I)$,
and we have the bound $\|u\|_{L^\infty(I)} + \|v\|_{L^\infty(I)} \le
C (1+d_R)$. Finally, since $\|u_x\|_{L^2(\R)} + \|v_x\|_{L^2(\R)}
\le C d_R$, we conclude that $u,v \in L^\infty(\R)$ and that
\eqref{Kbound1} holds.

\smallskip
\noindent{\bf Step 2:} Next, we assume that $\rho < \infty$, so that
$u_x,v_x \in H^1(\R)$, $u,v \in L^2(-R,R)$, and $\eta \in H^1(I_R)$,
where $I_R = \{x \in \R : |x| \ge R\}$. We claim that $u,v \in
L^\infty(\R)$ and that
\begin{equation}
\label{Kbound2}
  K \,:=\, \|u\|_{L^\infty(\R)} + \|v\|_{L^\infty(\R)} \,\le\, C(1 + R^{1/2}\rho),
\end{equation}
for some universal constant $C > 0$. Indeed, we know that $\eta
\in L^\infty(I_R)$ with $\|\eta\|_{L^\infty(I_R)} \le C\rho$. This
implies that $\psi \in L^\infty(I_R)$, hence also $u,v \in L^\infty(I_R)$,
and that $\|u\|_{L^\infty(I_R)}+\|v\|_{L^\infty(I_R)} \le C(1+\rho)^{1/2}$.
On the other hand, we know that $\|u\|_{L^\infty(-R,R)} \le C
\|u\|_{H^1(-R,R)} \le C\rho$ and that
$$
  \|v\|_{L^\infty(-R,R)} \,\le\, C\biggl(\frac{\|v\|_{L^2(-R,R)}}{R^{1/2}}
  + \|v\|_{L^2(-R,R)}^{1/2}\|v_x\|_{L^2(-R,R)}^{1/2}\biggr)
  \,\le\, CR^{1/2}\rho,
$$
because $\|v\|_{L^2(-R,R)} \le R\rho$ and $\|v_x\|_{L^2(-R,R)} \le \rho$.
Thus we conclude that $u,v \in L^\infty(\R)$ and that \eqref{Kbound2}
holds.

\smallskip
\noindent{\bf Step 3:} Finally we assume that $K = \|u\|_{L^\infty(\R)} +
\|v\|_{L^\infty(\R)} < \infty$, which is the case if $d_R < \infty$ or
if $\rho < \infty$. As $\eta = 2u_0u + u^2 + v^2$, we find
$$
  \|\eta\|_{L^2(-R,R)} \,\le\, C(1+K) \Bigl(\|u\|_{L^2(-R,R)} +
  \|v\|_{L^2(-R,R)}\Bigr) \,\le\, C(1+K)R\rho,
$$
because $\|u\|_{L^2(-R,R)} \le \rho$ and $\|v\|_{L^2(-R,R)} \le R\rho$.
This shows that, if $\rho < \infty$, then $\eta \in L^2(\R)$, so
that $d_R < \infty$, and we have the bound $d_R \le C(1+K)R\rho$.
Conversely, since $\eta_x = 2(u_0' u + u_0 u_x + uu_x + vv_x)$, we obtain
$$
  \|\eta_x\|_{L^2(\R)} \,\le\, C(1+K) \Bigl(\|u\|_{L^2(-1,1)} +
  \|u_x\|_{L^2(\R)} +  \|v_x\|_{L^2(\R)}\Bigr) \,\le\, C(1+K)d_R,
$$
where to estimate $u_0' u$ we used the fact that $|u(x)| \le
C(\|u\|_{L^2(-1,1)} + (1+|x|)^{1/2}\|u_x\|_{L^2(\R)})$. This shows that,
if $d_R < \infty$, then $\eta_x \in L^2(\R)$, so that
$\rho < \infty$, and we have the bound $\rho \le C(1+K) d_R$.
This concludes the proof.
\end{proof}

In the calculations below, to avoid boundary terms when integrating by
parts in expressions such as \eqref{Qdef}, it is technically
convenient to split the integration domain using a smooth partition of
unity. Let $\chi : \R \to [0,1]$ be a smooth cut-off function such
that
$$
  \chi(x) \,=\, 1 \quad \mbox{\rm for} \quad |x| \le \frac{1}{2}\,,
  \qquad {\rm and} \qquad \chi(x) \,=\, 0 \quad \mbox{\rm for} \quad
  |x| \ge \frac{3}{2}\,.
$$
We further assume that $\chi$ is even, that $\chi'(x) \le 0$ for
$x \ge 0$, and that $\chi(1) = \frac{1}{2}$. Given $R \ge 1$, we denote
$\chi_R(x) = \chi(x/R)$. The following estimates will be useful
to control the functions $u,v$ on the support of $\chi_R'$.

\begin{lemma}
\label{auxlem2}
Fix $R \ge 1$, and assume that $\psi = u_0 + u + iv$ satisfies
$d_R(\psi,u_0) < \infty$. Then there exists a constant $C_1 > 0$
(independent of $R$) such that
\begin{align}
\label{locbound1}
  \|u\|_{L^2(-2R,2R)} \,&\le\, C_1 (\rho(u,v,\eta) +
  R^{3/2}\rho(u,v,\eta)^2), \\ \label{locbound2}
  \|u\|_{L^\infty(-2R,2R)} + \|v\|_{L^\infty(-2R,2R)} \,&\le\,
  C_1 R^{1/2}\rho(u,v,\eta),
\end{align}
where $\rho(u,v,\eta)$ is given by \eqref{rhodef}.
\end{lemma}

\begin{proof}
If $f$ is either $u$ or $v$, then $|f(x)| \le
C(R^{-1/2}\|f\|_{L^2(-R,R)} + (|x|+R)^{1/2}\|f_x\|_{L^2(\R)})$, and
this gives the bound \eqref{locbound2}. To prove estimate
\eqref{locbound1}, we recall that $\|u\|_{L^2(-R,R)} \le
\rho(u,v,\eta)$, so we only need to control $u(x)$ for $R \le |x| \le
2R$. In that region we have $|u| \le C(|\eta| + u^2 + v^2)$, hence
using the bound \eqref{locbound2} and the fact that $\|\eta\|_{L^2(|x|
\ge R)} \le \rho(u,v,\eta)$ we obtain the desired result.
\end{proof}

We now analyze the quadratic terms in the representation \eqref{Qdef}.

\begin{lemma}
\label{lemma-soliton-4}
Under the assumptions of Lemma~\ref{auxlem}, if $d_R(\psi,u_0)
\le 1$, we have
\begin{equation}\label{Bident}
  \int_\R \Bigl(B_1(u) + B_3(\eta)\Bigr)\chi_R(x)\dd x \,=\,
  \int_\R B_0(u)\chi_R(x)\dd x + \mathcal{O}(R^3\rho(u,v,\eta)^3 +
  e^{-R}\rho(u,v,\eta)^2),
\end{equation}
where the estimate in the big O term holds uniformly for $R \ge 1$.
\end{lemma}

\begin{proof}
Since $\eta = 2u_0u + u^2 + v^2$, we find by a direct calculation
$$
  B_3(\eta) \,=\, 2u_0'^2u^2 + 2 u_0^2 u_x^2 + 4 u_0 u_0' u u_x +
  2(3u_0^2-2) u_0^2 u^2 + \tilde N(u,v),
$$
where
\begin{align*}
  \tilde N(u,v) \,&=\, 4(uu_x + vv_x)(u_0'u + u_0u_x) + 2(uu_x + vv_x)^2 \\
  &\quad + 4(3u_0^2-2)u_0u(u^2+v^2) + 2(3u_0^2-2)(u^2+v^2)^2.
\end{align*}
In view of the definitions \eqref{Bdef}, this implies that
$$
  B_1(u) + B_3(\eta) \,=\, B_0(u) + (2u_0 u_0' u^2)_x + \tilde N(u,v).
$$
If we now multiply both sides by
$\chi_R(x)$ and integrate over $x \in \R$, we arrive at \eqref{Bident},
because it is straightforward to verify using \eqref{rhodef},
\eqref{Kbound1} and \eqref{locbound2} that
$$
  -2\int_\R u_0 u_0' u^2 \chi_R'(x) \dd x \,=\, \mathcal{O}(e^{-R}
  \rho(u,v,\eta)^2), \quad \hbox{and}\quad \int_\R \tilde N(u,v)
  \chi_R(x) \dd x \,=\, \mathcal{O}(R^3\rho(u,v,\eta)^3).
$$
This concludes the proof of the lemma.
\end{proof}

Using Lemma~\ref{lemma-soliton-4}, we are able to derive the 
desired lower bound on the difference $\Lambda(\psi) - \Lambda(u_0)$ 
in terms of the quantity $\rho(u,v,\eta)$.

\begin{proposition}
\label{prop-soliton}
If $R \ge 1$ is sufficiently large, there exists a constant $C_2 > 0$
such that, if $\psi = u_0 + u + iv$ satisfies $d_R(\psi,u_0) \le 1$
and if $\langle u_0',u\rangle_{L^2} = \langle u_0'',v\rangle_{L^2} = 0$, then
\begin{equation}
\label{Lamlower}
  \Lambda(\psi) - \Lambda(u_0) \,\ge\, C_2 \rho(u,v,\eta)^2 +
  \mathcal{O}(R^3\rho(u,v,\eta)^3),
\end{equation}
where the estimate in the big O term is uniform in $R$.
\end{proposition}

\begin{proof}
Proceeding as in the proof of Lemma~\ref{auxlem}, it is easy to
estimate the cubic terms in \eqref{Lamexp} in terms of $\rho(u,v,\eta)$
using, in particular, the uniform bound \eqref{Kbound1} and
the estimate \eqref{locbound2}. We thus find
\begin{equation}
\label{lowbd0}
  \Lambda(\psi) - \Lambda(u_0) \,=\, Q(u,v,\eta)
  + \mathcal{O}(R^3\rho(u,v,\eta)^3),
\end{equation}
where $Q(u,v,\eta)$ is given by \eqref{Bdef} and \eqref{Qdef}.
Then, in the definition \eqref{Qdef}, we split the integral using
the partition of unity $1 = \chi_R + (1-\chi_R)$ and we use
Lemma~\ref{lemma-soliton-4}. This gives
\begin{align}\nonumber
  Q(u,v,\eta) \,&=\,
  \int_\R B_2(v)\dd x + \int_\R B_0(u)\chi_R(x)\dd x \\ \label{lowbd1}
  &\quad + \int_\R \Bigl(B_1(u) + B_3(\eta)\Bigr)(1-\chi_R(x))\dd x
  + \mathcal{O}(R^3\rho(u,v,\eta)^3 + e^{-R}\rho(u,v,\eta)^2).
\end{align}
As $\langle u_0'',v\rangle = 0$, we know from
Lemmas~\ref{operator-K-minus} and \ref{lemma-soliton-2} that
\begin{equation}
\label{lowbd2}
  \int_\R B_2(v)\dd x \,\ge\, C \int_\R (v_{xx}^2 + v_x^2)\dd x
  + \frac{C}{R^2}\int_{|x| \le R} v^2 \dd x,
\end{equation}
where the last term in the right-hand side follows from the bound
$|v(x)| \le |v(0)| + |x|^{1/2}\|v_x\|_{L^2}$, which implies
$$
  \int_{|x| \le R} v^2 \dd x \,\le\, 4R|v(0)|^2 + 2R^2 \int_\R v_x^2 \dd x
  \,\le\, C R^2  \int_\R B_2(v)\dd x.
$$
On the other hand, if $R \ge 1$ is large enough so that
$3u_0^2 - 2 \ge \frac{1}{2}$ for $|x| \ge R$, it is clear from
\eqref{Bdef} that
\begin{equation}
\label{lowbd3}
 \int_\R \Bigl(B_1(u) + B_3(\eta)\Bigr)(1-\chi_R(x))\dd x
 \,\ge\, C \int_{|x| \ge R} (u_{xx}^2 + u_x^2 + \eta_x^2 + \eta^2)\dd x.
\end{equation}

Finally, we estimate from below the term $\int_\R B_0(u)\chi_R(x)\dd
x$ under the orthogonality assumption $\langle u_0',u\rangle_{L^2} =
0$. Arguing as in Lemma~\ref{lemma-K-plus} and Corollary
\ref{lemma-K-minus}, we introduce the auxiliary variable $w = u_x +
\sqrt{2}u_0 u$. After integrating by parts, we obtain the identity
$$
  \int_\R B_0(u) \chi_R(x) \dd x = \int_\R \Bigl(w_x^2 + w^2
  \Bigr)\chi_R(x) \dd x + J_R,
$$
where
$$
 J_R \,=\, \int_\R \Bigl(\sqrt{2}u_0 u_x^2 + 2\sqrt{2}u_0' u u_x +
  (2u_0 u_0'- \sqrt{2}u_0'')u^2 + \sqrt{2}u_0^2 u^2\Bigr)
  \chi_R'(x)\dd x.
$$
Since $\chi_R'(x) = R^{-1} \chi'(x/R)$, we have using the estimate
\eqref{locbound1}
$$
  |J_R| \,\le\, \frac{C}{R} \int_{|x| \le 3R/2} \Bigl(u_x^2 + u^2\Bigr)\dd x
  \,\le\, \frac{C_3 \rho(u,v,\eta)^2}{R} + \mathcal{O}(R^2\rho(u,v,\eta)^4),
$$
where $C_3 > 0$ is independent of $R$. Moreover, proceeding as in
the proof of Lemma~\ref{lemma-soliton-1}, we find
\begin{equation}
\label{lowbdaux}
  \int_{|x| \le R} \Bigl(u_{xx}^2 + u_x^2 + u^2\Bigr)\dd x \,\le\,
  C \int_{|x| \le R} \Bigl(w_x^2 + w^2\Bigr)\dd x +
  \mathcal{O}(e^{-R}\rho(u,v,\eta)^2).
\end{equation}
Indeed, we have the representation $u = A u_0' + W$, where the
function $W$ is defined in \eqref{variation-u} and the constant $A$ is
fixed by the orthogonality condition $\langle u_0',u\rangle_{L^2} =
0$.  The proof of Lemma~\ref{lemma-soliton-1} shows that
$\|W\|_{L^2(|x|\le R)} \le C \|w\|_{L^2(|x|\le R)}$. From the
orthogonality relation
\[
  0 \,=\, \int_{|x|\le R} u_0'(x)\Bigl(A u_0'(x) + W(x)\Bigr)\dd x
  + \int_{|x|\ge R} u_0'(x) u(x)\dd x,
\]
we easily obtain the bound $|A| \le C\|W\|_{L^2(|x|\le R)} + \mathcal{O}(e^{-R}
\rho(u,v,\eta))$. This shows that
$$
  \|u\|_{L^2(|x|\le R)} \,\le\, C\|w\|_{L^2(|x|\le R)} +
  \mathcal{O}(e^{-R}\rho(u,v,\eta)),
$$
and since $u_x = w - \sqrt{2}u_0 u$ we obtain similar estimates for the
derivatives $u_x$ and $u_{xx}$, which altogether give \eqref{lowbdaux}.
Summarizing, we have shown
\begin{align}\nonumber
  \int_\R B_0(u) \chi_R(x) \dd x \,&\ge\, C \int_{|x| \le R} \Bigl(u_{xx}^2
  + u_x^2 + u^2\Bigr)\dd x - \frac{C_3 \rho(u,v,\eta)^2}{R} \\
  \label{lowbd4}
  &\quad\, + \mathcal{O}(R^2\rho(u,v,\eta)^3 + e^{-R}\rho(u,v,\eta)^2),
\end{align}
where in the big O term we replaced $R^2\rho(u,v,\eta)^4$ with
$R^2\rho(u,v,\eta)^3$ using the fact that $\rho(u,v,\eta) \le C_0
d_R(\psi,u_0) \le C_0$ by \eqref{rhoequiv}. Now, combining
\eqref{lowbd0}, \eqref{lowbd1}, \eqref{lowbd2}, \eqref{lowbd3},
\eqref{lowbd4}, and taking $R \ge 1$ sufficiently large, we arrive at
\eqref{Lamlower}.
\end{proof}

\begin{corollary}\label{Lambdafinal}
Fix any $R \ge 1$. There exist $\epsilon_1 \in (0,1)$ and $C_4 \ge 1$ such
that, if $\psi = u_0 + u + iv$ satisfies $d_R(\psi,u_0) \le \epsilon_1$
and if $\langle u_0',u\rangle_{L^2} = \langle u_0'',v\rangle_{L^2} = 0$,
then
\begin{equation}
\label{Lambdaest}
  C_4^{-1}d_R(\psi,u_0)^2 \le \Lambda(\psi) - \Lambda(u_0) \,\le\,
  C_4 d_R(\psi,u_0)^2.
\end{equation}
\end{corollary}

\begin{proof}
Choose $R \ge 1$ large enough so that the conclusion of
Proposition~\ref{prop-soliton} holds, and $\rho_0 > 0$ small enough
so that $R^3 \rho_0 \ll C_2$, where $C_2$ is as in \eqref{Lamlower}.
Take $\epsilon_1 \le 1$ such that $C_0\epsilon_1 \le \rho_0$, where
$C_0$ is as in \eqref{rhoequiv}. If $\psi = u_0 + u + iv$ satisfies
$d_R(\psi,u_0) \le \epsilon_1$ and $\langle u_0',u\rangle_{L^2} =
\langle u_0'',v\rangle_{L^2} = 0$, it follows from \eqref{rhoequiv}
that the quantity $\rho(u,v,\eta)$ defined in \eqref{rhodef}
satisfies $\rho(u,v,\eta) \le \rho_0$. By
Proposition~\ref{prop-soliton}, we thus have
$$
   \frac12 C_2 \rho(u,v,\eta)^2 \,\le\, \Lambda(\psi) - \Lambda(u_0)
  \le C_2' \rho(u,v,\eta)^2,
$$
where the lower bound follows from \eqref{Lamlower}, and the upper
bound can be established by a much simpler argument (which does
not use any orthogonality condition). Since $\rho(u,v,\eta)$ is
equivalent to $d_R(\psi,u_0)$ by Lemma~\ref{auxlem}, we obtain
\eqref{Lambdaest}. Finally, Corollary~\ref{Lambdafinal} holds
for any $R \ge 1$ because different values of $R$ give equivalent
distances $d_R$ on $X$.
\end{proof}

It is now easy to conclude the proof of Theorem~\ref{theorem-soliton}.
Fix any $R \ge 1$. Given any $\epsilon > 0$, we take
$$
  \delta \,=\, \frac{1}{2C_4}\,\min(2\epsilon,\epsilon_0,\epsilon_1),
$$
where $C_4 \ge 1$ and $\epsilon_1 > 0$ are as in
Corollary~\ref{Lambdafinal} and $\epsilon_0 > 0$ is as in
Lemma~\ref{lemma-xith}. If $\psi_0 \in X$ satisfies $d_R(\psi_0,u_0)
\le \delta$, then $\Lambda(\psi_0) - \Lambda(u_0) \le C_4
\delta^2$ by the upper bound in \eqref{Lambdaest}, which does not
require any orthogonality condition. Since $\Lambda$ is a conserved
quantity, we deduce that the solution $\psi(\cdot,t)$ of the cubic NLS
equation \eqref{nls} with initial data $\psi_0$ satisfies
$\Lambda(\psi(\cdot,t)) - \Lambda(u_0) \le C_4 \delta^2$ for all
$t \in \R$. We claim that, for all $t \in \R$, we have
\begin{equation}\label{inf3}
  \inf_{\xi, \theta \in \R} d_R\Bigl(e^{i \theta} \psi(\cdot + \xi,t),
  u_0\Bigr) \,\le\, 2C_4\delta \le \epsilon_0.
\end{equation}
Indeed, the bound \eqref{inf3} holds for $t = 0$ by assumption. Let
$\mathcal{J} \subset \R$ be the largest time interval containing the
origin such that the bound \eqref{inf3} holds for all $t \in
\mathcal{J}$. As is well-known \cite{Gerard,Zhidkov}, the solutions of
the cubic NLS equation \eqref{nls} with initial data in $X$ depend
continuously on time with respect to the distance $d_R(\psi,u_0)$.
This implies that the left-hand side of the bound \eqref{inf3} is a
continuous function of $t$, so that $\mathcal{J}$ is closed. On the
other hand, if $t \in \mathcal{J}$, then by Lemma~\ref{lemma-xith} we
can find $\xi,\theta \in \R$ such that the function $\tilde \psi(x) =
e^{i(\theta+t)}\psi(x+\xi,t)$ can be decomposed as in \eqref{decomp2}
with $u,v$ satisfying the orthogonality conditions
\eqref{projections2}. Applying Corollary~\ref{Lambdafinal} to $\tilde
\psi$, we deduce that
$$
  C_4^{-1}d_R(\tilde\psi,u_0)^2 \le \Lambda(\tilde\psi) - \Lambda(u_0)
  \,=\, \Lambda(\psi_0) - \Lambda(u_0)  \le C_4 \delta^2,
$$
so that $d_R(\tilde\psi,u_0) \le C_4\delta$. Using again a
continuity argument, we conclude that $\mathcal{J}$ contains
a neighborhood of $t$. Thus $\mathcal{J}$ is open, hence
finally $\mathcal{J} = \R$, so that the bound \eqref{inf3} holds for all
$t \in \R$. Using Lemma~\ref{lemma-xith}, we thus obtain
modulations parameters $\xi(t)$, $\theta(t)$ such that
$$
  d_R\Bigl(e^{i(\theta(t)+t)} \psi(\cdot + \xi(t),t)\,,u_0\Bigr)
  \,\le\, C_4 \delta \le \epsilon, \qquad t \in \R.
$$
Finally, Lemma~\ref{difflem} shows that the functions $\xi : \R \to
\R$ and $\theta : \R \to \R/(2\pi\Z)$ are continuously differentiable
and satisfy the bounds \eqref{bound-time-per}. The proof of
Theorem~\ref{theorem-soliton} is now complete.

\begin{remark}
Instead of introducing the auxiliary variable $\eta$ to cure the
imperfect decomposition \eqref{decomposition2}, it would be
advantageous to find a parametrization of the perturbations
that fully takes into account the geometry of the functional
$\Lambda$, and in particular the degeneracy of $\Lambda''(u_0)$.
Near the constant solution $u_1 \equiv 1$, it is most natural
to write $\psi(x,t) = (1 + r(x,t))e^{i\phi(x,t)}$, where $r$ and $\phi$
are real-valued functions. In that case, the usual energy
function \eqref{energy} allows us to control $r$ in
$H^1(\R)$ and $\phi_x$ in $L^2(\R)$.
In the same spirit, it is tempting to consider perturbations
of the black soliton of the form
\begin{equation}
\label{decomposition3}
  \psi(x,t) \,=\, (u_0(x) + r(x,t))e^{i\phi(x,t)}, \quad x \in \R,
\end{equation}
where $r,\phi$ are again real-valued functions. With this
representation, we find
\begin{equation}
\label{Lambdatry}
  \Lambda(\psi) - \Lambda(u_0) \,=\, \langle K_+r,r\rangle
  + \int_\R \Bigl(u_0^2 \phi_{xx}^2 + \phi_x^2)\Bigr)\dd x
  + \tilde N(r,\phi_x),
\end{equation}
where $\tilde N(r,\phi_x)$ collects the higher order terms.  This
formula is interesting, because it is not difficult to verify that
$\tilde N(r,\phi_x)$ can be controlled by the quadratic terms in
\eqref{Lambdatry} if $r$ is small in $H^2(\R)$ and $\phi_x$ small in
$H^1(\R)$. However, not all perturbations of the black soliton can be
written in the form \eqref{decomposition3} with $r,\phi$ satisfying
such smallness conditions, because $u_0$ vanishes at $x = 0$ in
\eqref{decomposition3}.
\end{remark}

\vspace{0.5cm}
\noindent{\bf Acknowledgement.} D.P. is supported by the
Chaire d'excellence ENSL/UJF.
He thanks members of Institut Fourier, Universit\'e Grenoble for
hospitality and support during his visit (January-June, 2014).

\end{document}